\documentclass[11pt, a4paper]{article}

\usepackage[utf8]{inputenc}
\usepackage[T1]{fontenc}

\usepackage{amsmath, amsfonts, amsthm, amssymb}   
\usepackage{mathtools}
\usepackage[italicdiff]{physics}
\usepackage{bm}
\usepackage[margin=1in]{geometry}
\usepackage{microtype}
\usepackage{booktabs}
\usepackage{enumitem}
\usepackage{parskip}
\usepackage{hyphenat}
\usepackage{subfigure}
\usepackage{subcaption}
\newtheorem{theorem}{Theorem}
\newtheorem{lemma}{Lemma}
\newtheorem{remark}{Remark}

\usepackage[style=numeric, sorting=nyt, backend=biber]{biblatex}
\usepackage[colorlinks=true, allcolors=blue]{hyperref}
\usepackage{cleveref}
\crefname{equation}{Equation}{Equations}

\newcommand{\bb}[1]{\mathbb{#1}}
\newcommand{\cl}[1]{\mathcal{#1}}
\newcommand{\f}[2]{\frac{#1}{#2}}
\newcommand{\oo}{\infty}
\newcommand{\hil}{\cl L^2_\rho(\Xi)}
\newcommand{\sob}{\cl H^s_\rho(\Xi)}
\DeclareMathOperator{\D}{D_{0,\textit{t}}^\alpha}
\usepackage{authblk}
\title{Stochastic Galerkin Method for Fractional Boundary Value Problems: Convergence Analysis and Numerical Treatment}
\date{}

\author[1]{Abhishek Kumar Singh} 
\author[2]{Dhyan Divakar Laad}
\author[2]{Vaibhav Mehandiratta\thanks{Corresponding author: \texttt{vaibhavm@goa.bits-pilani.ac.in}}}
\author[3]{Roland Pulch}

\affil[1]{Department of Mathematical Sciences, Indian Institute of Technology (Banaras Hindu University), Varanasi 221005, India}
\affil[2]{Department of Mathematics, Birla Institute of Technology and Science, Pilani, K K Birla Goa Campus, Zuarinagar, Sancoale, Goa 403726, India}
\affil[3]{Institute of Mathematics and Computer Science, University of Greifswald, Walther-Rathenau-Str. 47, 17489 Greifswald, Germany}

\begin{document}

\maketitle

\begin{abstract}
We study two-point fractional boundary value problems with uncertain input data, where randomness may enter through the coefficients and boundary conditions. To quantify the resulting uncertainty in the solution, we employ the generalized polynomial chaos (gPC) framework and develop a stochastic Galerkin formulation of the problem. A particular focus of this work is the convergence analysis of the resulting approximation. Rather than imposing assumptions directly on the stochastic coefficients appearing in the gPC representation, we introduce minimal regularity assumptions on the input data and use them to establish the properties required for the convergence analysis. Based on these results, we prove the convergence of the stochastic Galerkin approximation to the corresponding gPC solution. Numerical experiments are presented to illustrate the theoretical findings and to investigate the influence of random coefficients and boundary conditions on the statistical behavior of the solution.
\end{abstract}

\section{Introduction}
Boundary value problems (BVPs) arise naturally in the mathematical modeling of a wide range of phenomena in science and engineering. They play an important role in applications such as heat conduction, diffusion, fluid mechanics, elasticity, and electrostatics, where information about the solution is prescribed at different points of the domain. A classical example is the two-point BVP for ordinary differential equations (ODEs), as encountered in the static deflection of a uniform elastic beam, which is modeled by a fourth-order ODE~\cite{li2008unified}. However, in many applications, the underlying processes exhibit memory effects and non-local behaviour, which cannot be adequately described by classical differential equations. Differential equations involving non-local operators, such as fractional operators, provide a useful framework for modeling such phenomena~\cite{xu2013numerical,kumar2022numerical, mehandiratta2023well} . Among the most commonly used fractional operators are the Riemann--Liouville derivatives and Caputo derivatives. However, the Caputo derivative is particularly well suited to BVPs, as it allows boundary conditions to be prescribed in terms of classical (integer-order) derivatives, unlike the Riemann--Liouville derivative, which often leads to fractional-order initial or boundary data~\cite{almeida2016modeling,ford2011fractional}.


Deterministic models usually assume that the input data, including initial and boundary conditions, and material properties, are known exactly. In practical problems, however, this is rarely the case, as such quantities are often available only with some degree of uncertainty. To address the variability of model parameters, a common technique is the substitution of the parameters by random variables. Accounting for these uncertainties is therefore essential for obtaining reliable and physically meaningful predictions. Uncertainty quantification (UQ) is the branch of applied mathematics that is concerned with characterizing how uncertainties in the input data propagate through a model and affect its response. Two widely used approaches within this framework are Monte Carlo methods and gPC expansions~\cite{Robert2004,Xiu2003,sullivan2015,Xiu2010}. Monte Carlo methods are broadly applicable and straightforward to implement; however, achieving accurate statistical estimates typically requires a large number of samples due to their relatively slow convergence. In contrast, gPC methods can achieve significantly faster convergence when the solution depends smoothly on the random parameters.

The main idea of the gPC framework is to represent uncertain quantities in terms of orthogonal polynomials depending on random variables. Based on this representation, gPC methods can generally be divided into non-intrusive and intrusive approaches, depending on how the polynomial chaos coefficients are computed.  In non-intrusive methods~\cite{xiu2005high,babuvska2007stochastic,blatman2011adaptive}, the original deterministic problem is solved repeatedly for different realizations of the random variables using an existing deterministic solver, without reformulating the governing equations, whereas intrusive methods~\cite{xiu2009efficient,augustin2012stochastic} substitute a truncated gPC expansion into the governing equations and project the resulting residual onto a finite-dimensional polynomial chaos space. A well-known example of the intrusive approach is the stochastic Galerkin method. The term intrusive refers primarily to the reformulation of the original problem, although the practical implications of this reformulation depend on the type of equation under consideration. For instance, in the case of elliptic PDEs, an existing finite element solver typically cannot be reused to solve the resulting stochastic Galerkin system. Consequently, a separate solver or a suitable modification of the existing solver is often required, making the method intrusive also at the implementation level~\cite{eigel2014adaptive,gunzburger2014stochastic}. In contrast, for ODEs, differential-algebraic equations (DAEs), and fractional differential equations (FDEs), an existing time-stepping or predictor-corrector method can often be adapted more directly to solve the resulting stochastic Galerkin system. For example, in recent work on FDEs~\cite{pulch2025stochastic}, the predictor-corrector method was employed to solve both the original FDE and the corresponding stochastic Galerkin system. Both classes of methods have been applied to a variety of problems that involve ODEs, DAEs, and PDEs. Non-intrusive methods are generally easier to implement, since they require no modification of existing solvers, whereas intrusive methods such as the stochastic Galerkin approach can, under conditions favorable to gPC convergence, achieve comparable or better accuracy at a lower computational cost, at the expense of a more involved reformulation of the problem.

\subsection{Motivation and Contribution of the Paper}
Given the importance of non-local operators in capturing memory effects and the presence of random inputs in many physical systems, as discussed above, it is natural to study differential equations that incorporate both features. However, the literature on fractional differential equations with random parameters remains limited (see Table \ref{tab:litrev}). The studies listed in Table \ref{tab:litrev} consider initial value problems involving fractional operators and develop corresponding stochastic Galerkin systems for their numerical solution, along with the convergence results under certain assumptions.


Motivated by the limited literature and the importance of BVPs, in this paper, we study the interplay of UQ and two-point fractional BVPs (FBVPs). Such problems arise naturally, for instance, as steady state models of anomalous diffusion processes in porous and disordered media \cite{Jin2015}. Fractional derivatives are also used to model memory-dependent damping in viscoelastic beam vibrations \cite{DonmezDemir2012}, further motivating the need to understand how uncertainty in the input data propagates to the solution. Uncertainties may enter FBVPs through different sources, such as initial conditions, boundary conditions,  coefficients, or forcing terms. More specifically, we consider a linear two-point FBVP in which the coefficient of the fractional derivative depends on a random variable and the boundary conditions may also involve random variables, making the resulting problem a more realistic mathematical model.

\begin{table}
    \centering
    \footnotesize
    {\color{black}\begin{tabular}{p{0.5cm} p{4cm} p{9.5cm}}
         \hline
         \textbf{Year} &  \textbf{Author} &  \textbf{Results}\\ [1ex] \hline 
         2024 & Jornet \cite{jornet2024convergence} & Investigate the convergence of
the Galerkin projections for random non-linear FDEs\\ [1ex]  
2025 & Pulch and Singh \cite{pulch2025stochastic} & Proposed a stochastic Galerkin method for linear FDEs. \\ 
 \hline
    \end{tabular}}
     \caption{\textcolor{black}{Literature review of some recent research in the related domain.}}
     \label{tab:litrev}
\end{table}

To deal with these uncertainties, we employ the gPC framework together with the stochastic Galerkin method. This approach allows us to represent and propagate the uncertainty arising from the random coefficients and boundary conditions through the two-point FBVP. To the best of our knowledge, this is the first work that investigates two-point FBVPs involving random parameters using the stochastic Galerkin framework.

The main contributions of this paper are as follows:
\begin{itemize}
    \item We derive the stochastic Galerkin system for the two-point linear FBVPs with random parameters.
    \item In contrast to the existing works listed in Table \ref{tab:litrev}, where certain assumptions are imposed directly on the stochastic coefficients to establish the convergence of the stochastic Galerkin method, we impose minimum regularity assumptions only on the input data. Based on these assumptions, we establish the required properties of the stochastic coefficients and subsequently prove the convergence of the stochastic Galerkin method.
    \item The analysis establishes a direct connection between the regularity of the input data and the convergence of the stochastic Galerkin approximation, thereby avoiding additional a priori assumptions on the stochastic coefficients.
    \item Numerical experiments are presented to verify the theoretical convergence results and to demonstrate the effectiveness of the proposed stochastic Galerkin framework for FBVPs.
\end{itemize}

\subsection{Structure of the Paper}
The paper is organized as follows.
Section 2 introduces the fractional BVP and its stochastic reformulation using the gPC expansion and derives the corresponding stochastic Galerkin system. Section 3 establishes the convergence of the stochastic Galerkin approximation to the exact gPC solution under minimal regularity assumptions on the input data. Section 4 presents numerical experiments to validate the theoretical findings. Finally, Section 5 summarizes the main findings of this work and discusses possible directions for future research.

\section{Outline of the Problem}
This section describes the problem formulation and the stochastic Galerkin approach to find the approximate solution of the considered problem.

\subsection{Fractional Boundary Value Problem}
Let \(\alpha>0\) and \(f : [a, b] \to \bb R\) be such that \(f \in \cl{AC}^n[a, b]\), the space of functions with absolutely continuous \((n-1)\)-th derivatives, with \(n = \lceil \alpha \rceil\). Then, the Caputo fractional derivative of order \(\alpha\) of \(f\) is defined as
\[\D f(t) = \f 1{\Gamma (n - \alpha)} \int_0^t \f{f^{(n)}(\tau)}{(t - \tau)^{\alpha - n + 1}}\dd{\tau}.\]
The Caputo fractional derivative extends the familiar properties of the derivative to fractional orders, and thus, we may pose differential equations of fractional orders (FDEs) using the underlying derivative.

To this end, let \(T > 0\), \(\Xi \subset \bb R\) be a compact set and \(a : \Xi \to \bb R^+\) a measurable function (additional assumptions imposed on \(a\) are given in Section 3.1). Moreover, let \(y : [0, T] \times \Xi \to \bb R\) be a real-valued function, then for \(\alpha\in (1,2)\), we describe the following fractional boundary value problem (FBVP):
\begin{gather*}
\D y(t, \xi) = a(\xi)y(t, \xi), \tag{1} \label{fde}\\
y(0, \xi) = y_0(\xi), \quad  y(T, \xi) = y_1(\xi). \tag{2} \label{bound}
\end{gather*}

The general solution to the linear FDE \labelcref{fde} can be expressed in terms of the two parameter family of Mittag-Leffler functions as \cite[Theorem 5.12]{kilbas2006theory}
\[y(t, \xi) = c_0(\xi)E_{\alpha,1}(a(\xi)t^{\alpha}) + c_1(\xi)tE_{\alpha, 2}(a(\xi)t^{\alpha}) \tag{3} \label{sol}\]
where 
\[E_{\alpha, \beta}(z) = \sum_{k=0}^{\oo} \f{z^k}{\Gamma (\alpha k + \beta)}, \quad z \in \bb C, ~~\alpha,\beta>0\]
and the coefficient functions can be computed given boundary conditions \labelcref{bound} as
\[c_0(\xi) = y_0(\xi) \quad \text{and} \quad c_1(\xi) = \f{y_1(\xi) - y_0(\xi)E_{\alpha, 1}(a(\xi)T^\alpha)}{TE_{\alpha, 2}(a(\xi)T^\alpha)}.\]

\subsection{Stochastic Modeling}
We now model the parameter \(\xi\) considered in the problem \eqref{fde}-\eqref{bound} as a random variable \(\xi : \Omega \to \Xi\) on a probability space \((\Omega, \cl F, \mathbb P)\) as described in \cite[Section 2.2]{pulch2025stochastic}. Given a measurable function \(f : \Xi \to \bb R\), its expected value is given by
\[\mathbb E(f(\xi)) \coloneqq \int_{\Omega} f(\xi(\omega))\dd{\mathbb P(\omega)}.\]
Our study is restricted to the case of continuous probability measures \(P\), and as a consequence of the Radon-Nikodym theorem, the probability measure \(P\) admits a density function \(\rho\). Thus, it follows that
\[\mathbb E(f(\xi)) = \int_{\Xi} f(\xi')\rho(\xi')\dd{\xi'}.\]
We may now define an inner product
\[\langle f, g \rangle \coloneqq \mathbb E(f(\xi)g(\xi)) = \int_{\Xi} f(\xi')g(\xi')\rho(\xi')\dd{\xi'}\]
for two measurable functions \(f\) and \(g\), and the associated Hilbert space
\[\hil \coloneqq \{f : \Xi \to \bb R \mid \text{\(f\) is measurable and \(\mathbb E(f(\xi)^2) < \oo\)}\}\]
with norm \(\norm{f}_{\hil} = \sqrt{\langle f, f \rangle}\). 

An orthonormal system \(\{\Phi_n\}\) for the Hilbert space can be constructed consisting of polynomials \(\Phi_n : \Xi \to \bb R\) for each \(n \in \bb N_0\) via the Gram-Schmidt process. Without loss of generality, we assume the degree of \(\Phi_n\) is \(n\). While most distributions are associated with a complete set of basis polynomials, exceptions do exist (as is the case with the log-normal distribution). Nevertheless, in what follows, we assume that the distribution of \(\xi\) admits a complete basis. Furthermore, in the subsequent analysis, the coefficients are determined such that the residual of the resulting approximation is always orthogonal to the subspace spanned by \(\{\Phi_0, \Phi_1, \dots , \Phi_n\}\) for \(n \in \bb N_0\).

The generalized polynomial chaos (gPC) expansion of a function \(f \in \hil\) is given by
\[f(\xi) = \sum_{k=0}^\oo f_k\Phi_k(\xi) \]
where \(f_k = \langle f, \Phi_k\rangle\) for every \(k \in \bb N_0\). Since \(\xi\) is modeled to be a random variable, the solution \labelcref{sol} of the considered FBVP \eqref{fde}--\eqref{bound} becomes a stochastic process. Assuming \(y(t, \cdot) \in \hil\), it possesses a gPC expansion
\[y(t, \xi) = \sum_{k=0}^\oo \hat{v}_k(t)\Phi_k(\xi) \tag{4} \label{gpc}\]
for all \(t \in [0, T]\) with coefficient functions \(\hat{v}_k : [0, T] \to \bb R\) for every \(k \in \bb N_0\). By defining the \(n\)th partial sum of \labelcref{gpc} as
\[\hat{y}^{(n)}(t, \xi) \coloneqq \sum_{k=0}^n \hat{v}_k(t)\Phi_k(\xi),\tag{5} \label{trunc}\]
one has
\[\lim_{n \to \oo} \norm{y(t, \cdot) - \hat{y}^{(n)}(t, \cdot)}_{\hil} = 0, ~\text{for all}~ t \in [0, T] \tag{6} \label{trueconv}.\]

\subsection{Stochastic Galerkin Method}\label{sgmethod}
The stochastic Galerkin approach is a spectral method used to solve a variety of random differential equations. Unlike non-intrusive sampling methods like Monte Carlo simulations, the stochastic Galerkin method incorporates the uncertainty directly into the solution by representing it as a gPC expansion. 
The resulting system consists of deterministic equations that can be solved numerically. 

Our objective is to compute approximations of the coefficient functions in the truncated series \labelcref{trunc}. Let \(\vb v^{(n)} = [v_0, v_1, v_2, \dotsc, v_n]^\top\) be a vector of approximate coefficients \(v_k \approx \hat{v}_k\) for \(k \in \{0, 1, \dots , n\}\). Define the approximate solution
\[\tilde{y}^{(n)}(t, \xi) \coloneqq \sum_{k=0}^n v_k(t)\Phi_k(\xi). \tag{7} \label{apx}\]
Inserting the approximation \labelcref{apx} into the FDE \labelcref{fde} yields a residual \(r^{(n)}(t, \xi)\):
\begin{align*}
r^{(n)}(t, \xi) &= \D \tilde{y}^{(n)}(t, \xi) - a(\xi)\tilde{y}^{(n)}(t, \xi) \\
&= \sum_{k=0}^n (\D v_k(t))\Phi_k(\xi) - a(\xi)\sum_{k=0}^n v_k(t)\Phi_k(\xi).
\end{align*}
Taking the inner product on \(\hil\) with \(\Phi_j\) and using the fact that the residual has to be orthogonal to the subspace, we have
\[\D v_j(t) - \sum_{k=0}^n v_k(t)\langle a\Phi_k, \Phi_j\rangle = 0, \quad j \in \{0, 1, \dots , n\}. \tag{8} \label{syseq}\]
Thus, the system \labelcref{syseq} can be represented in matrix-vector form as
\[\D \vb v^{(n)}(t) = A\vb v^{(n)}(t) \tag{9} \label{sysmat}\]
where the entries of the matrix \(A = [a_{ij}] \in \bb R^{(n+1) \times (n+1)}\) are given by
\[a_{ij} = \langle a\Phi_i, \Phi_j \rangle = \int_{\Xi} a(\xi')\Phi_i(\xi')\Phi_j(\xi')\rho(\xi')\dd{\xi'}.\]
The boundary conditions for the system are determined using the gPC expansion \labelcref{gpc} and the values of the approximate coefficient functions at the boundary coincide with the true values. Hence, for \(i \in \{0, \dots , n\}\), we have
\begin{align*}
\begin{split}
v_i(0) &= \hat{v}_i(0) = \langle y_0, \Phi_i \rangle = \int_{\Xi}y_0(\xi')\Phi_i(\xi')\rho(\xi')\dd{\xi'}, \\
v_i(T) &= \hat{v}_i(T) = \langle y_1, \Phi_i \rangle = \int_{\Xi}y_1(\xi')\Phi_i(\xi')\rho(\xi')\dd{\xi'}.
\end{split}
\tag{10} \label{sysbound}
\end{align*}
These equations require the additional assumption that the functions \(y_0\) and \(y_1\) are elements of \(\hil\). Finally, one can observe that Equations \eqref{sysmat}--\eqref{sysbound} represent a FBVP of a deterministic system which can be solved using the existing numerical methods.
\section{Convergence of the Stochastic Galerkin Method}
This section details the proof for convergence of the Galerkin system \eqref{sysmat}--\eqref{sysbound} to the gPC expansion \labelcref{gpc}.

\subsection{Auxiliary Results}\label{prelim}
We start by providing the following result from \cite[Theorem 7.32.1]{szego1939orthogonal} that gives the uniform bound for the orthonormal basis polynomials on a compact interval.

\begin{lemma}\label{phinbound}
Let \(\{\Phi_n\}_{n \in \mathbb{N}}\)
be a sequence of orthonormal basis polynomial in $ \hil$.
Then, there exist constants \(C > 0\) and \(\delta \geq 0\) such that \(\norm{\Phi_n}_\oo \leq Cn^\delta\) for all \(n \in \bb N\).
\end{lemma}

\begin{remark}   
   We note that while no formal upper bound for \(\delta\) in the above lemma exists, in general, its values are typically small in practice. In particular, for widely used orthonormal Legendre and Chebyshev polynomials (of the first kind), \(\delta\) takes values \(1/2\) and \(0\), respectively. More generally, for the Jacobi polynomials having parameters \(\alpha\) and \(\beta\) with $\alpha,\beta>-1$, one has \(\delta = \max\{\alpha, \beta\} + 1/2\), which remains moderate for standard parameter ranges. Consequently, for most classical polynomial bases associated with bounded continuous weight functions in the Askey scheme, considered here, the value of \(\delta\) is less than or equal to \(1\) under commonly used parameter choices. A notable exception is the Gegenbauer family of polynomials, for which \(\delta\) takes the same value as the free parameter \(\lambda\).
\end{remark}
   
The preceding lemma includes only the cases for \(n \geq 1\). However, as per our assumption on \(\{\Phi_n\}\), \(\Phi_0\) is a degree \(0\) polynomial and can be bounded by a constant. Furthermore, in the cases of the aforementioned Jacobi and Gegenbauer polynomials, the orthonormality condition always implies \(\Phi_0(x) \equiv \pm 1\).

Next, to provide the regularity of the solution of the considered FBVP \eqref{fde}--\eqref{bound}, we define the weighted Sobolev space \(\sob\), for \(s=0,1,2,\dotsc,\)
\[\sob:= \left\{ f: \Xi \to \mathbb{R} : \dv[m]{f}{x}\in \cl L^2_\rho(\Xi),~0\leq m\leq s\right\},\]
equipped with an inner product
\[\langle f,g\rangle_{\sob}:=\sum_{m=0}^{s}\left \langle \dv[m]{f}{x}, \dv[m]{g}{x}\right \rangle.\]
We now impose the following hypothesis on the given data (functions) that are required for the subsequent convergence analysis of the proposed method. 
\begin{enumerate}[label=(H)]
\item [(H1)] In the boundary conditions \labelcref{bound}, the functions satisfy \(y_0, y_1 \in \sob\) with \(s>\delta + 2\).
\item [(H2)] The measurable function \(a\in \cl W^{s,\infty}(\Xi)\) is essentially bounded by a constant \(C_a > 0\), where \(\cl W^{s,\infty}\) denotes the standard Sobolev space defined by
\[\cl W^{s,\infty}(\Xi):=\left\{a:\Xi \to \mathbb{R}: \dv[m]{a}{x}\in \cl L^{\infty}(\Xi), 0\leq m\leq s\right\}.\]\label{h2}
\end{enumerate}

\begin{lemma}\label{yregularity}
Given \emph{(H1)--(H2)}, the solution \(y(t, \cdot)\) of the FBVP \eqref{fde}--\eqref{bound} belongs to the weighted Sobolev space \(\sob\) for all \(t\in [0,T]\).
\end{lemma}

\begin{proof}
Since \(a\in \cl W^{s,\infty}(\Xi)\) and the Mittag-Leffler function \(E_{\alpha,\beta}\) is entire, it follows that
\(E_{\alpha,\beta}(a( \cdot )t^\alpha)\in \cl W^{s,\infty}( \Xi )\) for all \(t\in [0,T]\). Furthermore, in view of the multiplication property of the Sobolev spaces \cite{behzadan2021multiplication}, we have that \(y_0(\cdot) E_{\alpha,1}(a(\cdot)t^\alpha)\in \cl H_\rho^s(\Xi)\) for all \(t\in [0,T]\).

Now, since \(0 \neq E_{\alpha,2}(a(\cdot))\in \cl W^{s,\infty}(\Xi)\), using the same argument as above, we have that
\[
c_1(\cdot)
=\f{1}{TE_{\alpha,2}(a(\cdot)T^\alpha)}{\big(y_1(\cdot)-y_0(\cdot)E_{\alpha,1}(a(\cdot)T^\alpha)\big)}
\in \cl H_{\rho}^s(\Xi).
\]
Again, the multiplication property provides that
\(c_1(\cdot)tE_{\alpha,2}(a(\cdot)t^\alpha)\in \cl H_{\rho}^s(\Xi)\), which shows that \(y(t,\cdot)\in \cl H_{\rho}^s(\Xi)\).
\end{proof}

\begin{remark}
Notably, we only assume that the input functions belong to the Sobolev space \(\sob\). Under this minimal regularity assumption, we establish the uniform convergence of the gPC expansion and the superlinear convergence condition, as stated in Lemma~\ref{uniconv} and Lemma~\ref{superlinear}, respectively. In contrast to \cite[Section 3.1]{pulch2025stochastic}, where these properties are imposed as assumptions to derive the convergence analysis, our framework derives them directly from the regularity of the input data. Consequently, the assumptions required for the convergence analysis of the proposed stochastic Galerkin method are significantly relaxed.
\end{remark}


We now state an auxiliary result required to prove that the series \labelcref{gpc} and its Caputo derivative are uniformly convergent on \([0,T]\) for each \(\xi\in \Xi\).
\begin{lemma}\label{auxiliary}
Let \(\hat{v}_k\) denote the coefficients of the gPC expansion \eqref{gpc}. Then, there exists constants \(C>0\) and \(s > 0\) such that
\[
|\hat{v}_k(t)| \le C k^{-s}, \quad k \in \mathbb{N},
\]
for all \(t \in [0,T]\).
\end{lemma}
\begin{proof}
In view of \cite[Theorem 3.6]{xiu2010numerical}, there exists \(K > 0\), independent of \(n\), such that
\[\sum_{k=n}^\oo \hat{v}_k(t)^2 
\leq Kn^{-2s}\norm{y(t, \cdot)}_{\sob}^2 \]
for all \(t \in [0, T]\).

Now, for any fixed \(\xi\) and \(t \in [0,T]\), the map \(t \mapsto E_{\alpha,\beta}(a(\xi)t^\alpha)\) is continuous for any $\alpha,\beta>0$. Moreover, in view of \eqref{sol}, the solution \(y(t,\xi)\) can be expressed as a linear combination of terms involving multiplication by functions of the form \(E_{\alpha,\beta}(a(\xi)t^\alpha)\). Since \(a \in \cl W^{s,\infty}(\Xi)\) and the Mittag–Leffler functions are smooth, it follows that these multipliers belong to \(\cl W^{s,\infty}(\Xi)\) and depend continuously on \(t\). Hence, using Lemma \ref{yregularity} and the continuity of multiplication in Sobolev spaces, we obtain that
$
t \mapsto y(t,\cdot) \in \sob
$
is continuous on \([0,T]\). Consequently, the map \(t \mapsto \|y(t,\cdot)\|_{\sob}\) is continuous on the compact interval \([0,T]\), and therefore
$
y \in C([0,T]; \sob).
$
In particular, we have
\[
M := \sup_{t \in [0,T]} \|y(t,\cdot)\|_{\sob} < \infty.
\]

Then, it follows that for all \(k \in \bb N\)
\[\hat{v}_k(t)^2 \leq \sum_{j=k}^{\infty}\hat{v}_j(t)^2 \leq KM^2k^{-2s}.\]
Finally, choosing \(C = \sqrt{K}M\), we get,
\[\abs{\hat{v}_k(t)} \leq Ck^{-s},\]
uniformly for \(t \in [0, T]\).
\end{proof}

We now present the proof for uniform convergence of the gPC expansion.
\begin{lemma}\label{uniconv}
Under the hypotheses \emph{(H1)--(H2)}, the gPC expansion \labelcref{gpc}
\[y(t, \xi) = \sum_{k=0}^\oo \hat{v}_k(t) \Phi_k(\xi)\]
converges uniformly in \(t \in [0, T]\) for each \(\xi \in \Xi\). Furthermore, the series
of the term-by-term Caputo derivatives converges uniformly in \(t \in [0, T]\) to the Caputo derivative of the solution, i.e.,
\[\D y(t, \xi) = \sum_{k=0}^\oo (\D \hat{v}_k(t))\Phi_k(\xi).\]
\end{lemma}
\begin{proof}
To prove the uniform convergence of the gPC expansion on \([0,T]\), we employ the Weierstrass M-test. Fixing \(\xi \in \Xi\) and considering the supremum of the \(k\)th term in the expansion over \(t\):
\[\sup_{t \in [0, T]} \abs{\hat{v}_k(t)\Phi_k(\xi)} = \abs{\Phi_k(\xi)} \sup_{t \in [0, T]} \abs{\hat{v}_k(t)}.\]
In view of Lemma \(\ref{auxiliary}\), there exists \(C_1 > 0\) such that \(\abs{\hat{v}_k(t)} \leq C_1k^{-s}\) uniformly for \(t \in [0, T]\). Moreover, by Lemma \(\ref{phinbound}\), \(\abs{\Phi_k(\xi)} \leq \norm{\Phi_k}_\oo \leq C_2k^\delta\) for some \(C_2 > 0\). Thus,
\[\sup_{t \in [0, T]} \abs{\hat{v}_k(t)\Phi_k(\xi)} \leq C_2k^\delta \cdot C_1k^{-s} = C_1C_2k^{\delta - s}.\]
Thus, it holds that the gPC expansion converges uniformly on \([0,T]\) as \(\delta-s<-1\) due to (H1).

We now proceed to show that the series of the term-by-term Caputo derivatives converges uniformly on \([0,T]\) to the Caputo derivative of the solution. Firstly, recall the FDE \labelcref{fde}:
\[\D y(t, \xi) = a(\xi)y(t, \xi).\]
We know that for any \(k \in \bb N_0\)
\[\D \hat{v}_k(t) = \D \langle y(t, \cdot), \Phi_k\rangle.\]
Using the definition of the inner product, the Leibniz rule and the Fubini theorem, we get
\begin{align*}
    \D \langle y(t, \cdot), \Phi_k \rangle &= \D \int_\Xi y(t, \xi') \Phi_k(\xi') \rho(\xi') \dd{\xi'} \\
    &= \f{1}{\Gamma(2-\alpha)} \int_0^t (t-s)^{1-\alpha} \left( \int_\Xi \pdv[2]{s} y(s, \xi') \Phi_k(\xi') \rho(\xi') \dd{\xi'} \right) \dd{s} \\
    &= \int_\Xi \left( \f{1}{\Gamma(2-\alpha)} \int_0^t (t-s)^{1-\alpha} \pdv[2]{s} y(s, \xi') \dd{s} \right) \Phi_k(\xi') \rho(\xi') \dd{\xi'} \\
    &= \int_\Xi \left( \D y(t, \xi') \right) \Phi_k(\xi') \rho(\xi') \dd{\xi'} \\
    &= \langle \D y(t, \cdot), \Phi_k \rangle \\
    &= \langle a(\cdot)y(t, \cdot), \Phi_k \rangle.
\end{align*}

Therefore, it follows that \(\D \hat{v}_k(t)\) is the \(k\)th coefficient in the gPC expansion of \(g(t, \xi) = a(\xi)y(t, \xi)\).
Thus, by applying Lemma 3 for the function \(g(t, \cdot)\), we get 
\[\sup_{t \in [0, T]} \abs{\D \hat{v}_k(t)} \leq C_3k^{-s}\]
for some constant \(C_3>0\). Therefore,
\[\sup_{t \in [0, T]} \abs{(\D \hat{v}_k(t)) \Phi_k(\xi)} \leq C_2C_3k^{\delta - s}.\]
Again, since \(\delta - s < -1\), the sequence of bounds is summable. Consequently, the series of term-by-term Caputo derivatives converges uniformly on \([0,T]\).
\end{proof}

The following lemma was stated as assumption (A3) in \cite[Section 3.1]{pulch2025stochastic}, which we can now derive as a consequence of Lemma \ref{auxiliary}.
\begin{lemma}\label{superlinear}
The coefficients of the gPC expansion \labelcref{gpc} satisfy the following property:
\[\lim_{n \to \oo} \max_{t \in [0, T]} (n+1)\sum_{k=n+1}^\oo \abs{\hat{v}_k(t)} = 0.\]
\end{lemma}
\begin{proof}
By Lemma \ref{auxiliary}, there exist constants \(C > 0\) and \(s > 0\) such that
\[\abs{\hat{v}_k(t)} \leq Ck^{-s}.\]
Thus, we obtain
\[\sum_{k=n+1}^\oo \abs{\hat{v}_k(t)} \leq C\sum_{k=n+1}^\oo k^{-s}.\]
In view of (H1), we have \(s > 2\), and thus, we can bound the summation with the corresponding integral to get
\[\sum_{k=n+1}^\oo k^{-s} \leq \int_n^\oo x^{-s}\dd{x} = \f{n^{-(s-1)}}{s-1}.\]
Hence,
\[0 \leq (n+1)\sum_{k=n+1}^\oo \abs{\hat{v}_k(t)} \leq (n+1) \cdot \f{Cn^{-(s-1)}}{s-1}.\]
By the Squeeze Theorem \cite[Theorem 3.2.7]{bartle2011introduction},
\[\lim_{n \to \oo} (n+1) \sum_{k=n+1}^\oo \abs{\hat{v}_k(t)} = 0, \quad \forall t\in[0,T],\]
and as such,
\[ \lim_{n \to \oo} \max_{t \in [0, T]} (n+1)\sum_{k=n+1}^\oo \abs{\hat{v}_k(t)} = 0. \]
\end{proof}

Finally, before embarking on the convergence analysis of the proposed method, we provide the following lemmata from \cite[Section 3.1]{pulch2025stochastic}.

\begin{lemma}\label{abound}
Given (H2) and a constant \(C_a > 0\), the entries of the matrix \(a_{ij}\) in the system \labelcref{sysmat}--\labelcref{sysbound} satisfy \(\abs{a_{ij}} \leq C_a\) for all \(i, j \in \{0, 1, 2, \dots , n\}\).
\end{lemma}

\begin{lemma}\label{lambdabound}
Given (H2) and a constant \(C_a > 0\), each eigenvalue \(\lambda\) of the matrix in the system \labelcref{sysmat}--\labelcref{sysbound} satisfies \(\abs{\lambda} < C_a\).
\end{lemma}

\subsection{Convergence Analysis}\label{converge}
We now proceed to show the convergence of the solution satisfying the Galerkin system \labelcref{sysmat}--\labelcref{sysbound} to the solution of the FBVP \eqref{fde}--\eqref{bound}.

\begin{theorem}
Let \(y(t,\xi)\) be the solution of FBVP \eqref{fde}--\eqref{bound}, whose exact gPC expansion is given by \cref{gpc} and let \(\tilde{y}^{(n)}(t,\xi)\) be the \((n+1)\)th term gPC solution generated by the Galerkin system \labelcref{sysmat}--\labelcref{sysbound}. Then, it holds that
\[\lim_{n \to \oo} \norm{y - \tilde{y}^{(n)}}_{C([0,T];\hil)} = \lim_{n \to \oo} \max_{t \in [0,T]} \norm{y(t,\cdot) - \tilde{y}^{(n)}(t,\cdot)}_{\hil} = 0.\]
\end{theorem}

\begin{proof}
We begin by applying the triangle inequality to the error norm in the Hilbert space \(\hil\), using the truncated gPC expansion \labelcref{trunc} and the approximate solution \labelcref{apx}:
\[\norm{y(t, \cdot) - \tilde{y}^{(n)}(t, \cdot)}_{\hil} \leq \norm{y(t, \cdot) - \hat{y}^{(n)}(t, \cdot)}_{\hil} + \norm{\hat{y}^{(n)}(t, \cdot) - \tilde{y}^{(n)}(t, \cdot)}_{\hil}.\]
By Lemma \labelcref{uniconv} the gPC expansion converges uniformly for \(t \in [0, T]\):
\[\lim_{n \to \infty} \max_{t \in [0, T]} \norm{y(t, \cdot) - \hat{y}^{(n)}(t, \cdot)}_{\hil} = 0.\] Consequently, it suffices to demonstrate that the second term, representing the error between the projection of the exact solution and the Galerkin approximation, vanishes, i.e.,
\[\lim_{n \to \infty} \norm{\hat{y}^{(n)}(t, \cdot) - \tilde{y}^{(n)}(t, \cdot)}_{\hil} = 0\]
uniformly for all \(t \in [0, T]\).

In view of Parseval's identity, it follows that
\[\norm{\hat{y}^{(n)}(t, \cdot) - \tilde{y}^{(n)}(t, \cdot)}^2_{\hil} = \norm{\sum_{k=0}^n (\hat{v}_k(t) - v_k(t))\Phi_k(\cdot)}_{\hil}^2 = \sum_{k=0}^n (\hat{v}_k(t) - v_k(t))^2.\]
Let \(e_k(t) = \hat{v}_k(t) - v_k(t)\) denote the error in the \(k\)th coefficient, and define the error vector \(\vb e(t) = [e_0(t) \; e_1(t) \; \cdots \; e_n(t)]^\top\). Then, we must show that \(\norm{\vb e(t)}_2 \to 0\) as \(n\to \infty\). Inserting the exact gPC expansion into the FDE \labelcref{fde} yields
\[ \D \left( \sum_{k=0}^\oo \hat{v}_k(t) \Phi_k(\xi) \right) = a(\xi) \sum_{k=0}^\oo \hat{v}_k(t) \Phi_k(\xi). \]
As established in Lemma~\ref{uniconv}, the series of term-by-term Caputo derivatives converges uniformly. This permits the interchange of the fractional derivative operator and the infinite summation. Taking the inner product of both sides with an arbitrary basis polynomial \(\Phi_j\) on \(\hil\) and interchanging the summation and the inner product (which holds since the series converges in \(\hil\)), we obtain
\[ \sum_{k=0}^\oo (\D \hat{v}_k(t)) \langle \Phi_k, \Phi_j \rangle = \sum_{k=0}^\oo \hat{v}_k(t) \langle a \Phi_k, \Phi_j \rangle. \]
Applying the orthonormality of the basis polynomials, \(\langle \Phi_k, \Phi_j \rangle = \delta_{kj}\), the left-hand side simplifies, resulting in the relation
\[\D \hat{v}_j(t) = \sum_{k=0}^\infty a_{kj}\hat{v}_k(t),\]
where \(a_{kj} = \langle a \Phi_k, \Phi_j \rangle\). Applying the derivative operator to the error equation and substituting the relation above yields
\begin{align*}\D e_j(t)& = \sum_{k=0}^\infty a_{jk}\hat{v}_k(t) - \sum_{k=0}^n a_{kj}v_{k}(t)\\
&= \sum_{k=0}^n a_{jk}e_k(t) + \sum_{k=n+1}^\infty a_{jk}\hat{v}_k(t).
\end{align*}
We define the vector \(\hat{\vb R}(t) = [\hat{R}_0(t) \; \hat{R}_1(t) \; \cdots \; \hat{R}_n(t)]^\top\), with components given by
\[\hat{R}_j(t) = \sum_{k=n+1}^\infty a_{jk}\hat{v}_k(t).\]
Consequently, the system governing the error vector can be written as
\[\D \vb e(t) = A \vb e(t) + \hat{\vb R}(t).\]
Since the matrix \(A\) is symmetric, it admits the spectral decomposition \(A = P\Lambda P^\top\). We introduce the transformed error vector
\[\vb d(t) = P^\top \vb e(t).\]
Substituting this transformation into the system yields
\[\D \vb d(t) = \Lambda \vb d(t) + P^\top\hat{\vb R}(t).\]
Since \(\Lambda\) is a diagonal matrix, the system decouples into \(n+1\) independent scalar FDEs:
\[\D d_k(t) = \lambda_k d_k(t) + R_k(t),\]
where \(\vb R(t) = P^\top \hat{\vb R}(t)\). The general solution to this differential equation is expressed in terms of Mittag-Leffler functions as
\[d_k(t) = c_{k,1}E_{\alpha, 1}(\lambda_kt^\alpha) + c_{k,2}tE_{\alpha, 2}(\lambda_k t^\alpha) + \int_0^t (t - \tau)^{\alpha - 1}E_{\alpha, \alpha} (\lambda_k(t - \tau)^\alpha)R_k(\tau)\dd{\tau}.\]
Imposing the homogeneous boundary conditions \(\vb d(0) = P^\top \vb e(0) = \vb 0\) and \(\vb d(T) = P^\top \vb e(T) = \vb 0\), we determine the constants:
\[c_{k,1} = 0, \quad c_{k,2} = -\f{1}{TE_{\alpha, 2}(\lambda_kT^\alpha)} \int_0^T (T - \tau)^{\alpha - 1} E_{\alpha, \alpha} (\lambda_k(T - \tau)^\alpha)R_k(\tau)\dd{\tau}.\]
We now proceed to derive a bound for the function \(d_k\). By (H2), the function \(a\) is essentially bounded by the constant \(C_a > 0\). Thus, it follows that the eigenvalues satisfy \(\abs{\lambda_k(t - \tau)^\alpha} \leq C_a T^\alpha\) for all \(\tau \in [0, t]\) and \(t \in (0, T]\) by Lemma \ref{lambdabound}. As the function \(E_{\alpha, \alpha}\) is continuous, we define
\[C = \max_{\abs{x} \leq C_a T^\alpha} \abs{E_{\alpha, \alpha}(x)}.\]
The integral term is bounded as follows:
\begin{align*}
\abs{\int_0^t (t - \tau)^{\alpha - 1}E_{\alpha, \alpha}(\lambda_k(t - \tau)^\alpha)R_k(\tau)\dd{\tau}} &\leq \int_0^t \abs{(t - \tau)^{\alpha - 1}E_{\alpha, \alpha}(\lambda_k(t - \tau)^\alpha)R_k(\tau)}\dd{\tau}\\
&\leq \f{t^\alpha}{\alpha} \cdot C \cdot \norm{R_k}_\infty \le C\f{T^{\alpha}}{\alpha}\norm{R_k}_\infty.
\end{align*}
Similarly, the numerator of \(c_{k,2}\) corresponds to the integral evaluated at \(t=T\), and thus, we get
\[\abs{c_{k,2}tE_{\alpha,2}(\lambda_kt^\alpha)} \leq \f{1}{\abs{E_{\alpha, 2}(\lambda_kT^\alpha)}} \cdot C\cdot \f{T^\alpha}{\alpha} \norm{R_k}_{\infty} \cdot \max_{t \in [0, T]} \abs{E_{\alpha, 2}(\lambda_k t^\alpha)}.\]
Combining these results, the bound for \(d_k(t)\) is given by
\[|d_k(t)| \leq C \f{T^\alpha}{\alpha} \left ( 1 + \f{\max_{t \in [0,T]}|E_{\alpha, 2}(\lambda_kt^\alpha)|}{|E_{\alpha, 2}(\lambda_kT^\alpha)|}\right ) \norm{R_k}_\infty.\]
Define a uniform constant \(K\):
\[K \coloneqq \f{T^\alpha}{\alpha}\left ( 1 + \f{\max_{\abs{t} \leq C_aT^\alpha} \abs{E_{\alpha, 2} (t)}}{\min_{\abs{t} \leq C_aT^\alpha} \abs{E_{\alpha, 2} (t)}}\right ). \]
Note that \(K\) is independent of \(\lambda_k\) since \(\abs{\lambda_kt^{\alpha}} \leq C_aT^\alpha\). Therefore, it holds that
\[|d_k(t)| \leq K\norm{R_k}_\infty.\]
We now bound the residual vector. Since \(P\) is orthogonal, \(P^\top P = I\), implying that the Euclidean norm of each column vector \(\vb p_i\) is unity. By the Cauchy-Schwarz inequality in Euclidean space,
\[|R_i(t)| = |\vb p_i^\top \hat{\vb R}(t)| \leq \norm{\vb p_i}_2\norm{\hat{\vb R}(t)}_2 = \left ( \sum_{\ell = 0}^n \left (\sum_{k=n+1}^\infty a_{\ell k}\hat{v}_k(t)\right )^2\right )^{1/2}.\]
In view of Lemma 6, we have \(|a_{\ell k}| \leq C_a\). Therefore,
\[\norm{\vb R(t)}_\infty \leq \left ( \sum_{\ell = 0}^n \left ( C_a \sum_{k=n+1}^\infty |\hat{v}_k(t)|\right )^2\right )^{1/2} = C_a \sqrt{n+1} \sum_{k=n+1}^\infty |\hat{v}_k(t)|.\]
Finally, we relate the error vector \(\vb e(t)\) back to the transformed coordinates. Since the orthogonal transformation preserves the Euclidean norm, \(\norm{\vb e(t)}_2 = \norm{\vb d(t)}_2\), in view of the equivalence of the Euclidean norm and maximum norm in \(\mathbb{R}^{n+1}\), we get
\[\norm{\vb e(t)}_2 \leq \sqrt{n+1} \norm{\vb d(t)}_\infty.\]
Substituting the derived bounds yields
\[\norm{\vb e(t)}_2 \leq \sqrt{n+1} K \norm{\vb R}_\infty \leq K C_a (n+1) \sum_{k=n+1}^\infty |\hat{v}_k(t)|.\]
Taking the limit as \(n \to \infty\), Lemma \ref{superlinear} ensures that the right-hand side converges uniformly to zero. Therefore,
\[\lim_{n \to \infty} \max_{t \in [0, T]} \norm{\vb e(t)}_2 = 0.\]
This concludes the proof.
\end{proof}

\section{Numerical Experiments}
In this section, we demonstrate the performance of the proposed stochastic Galerkin method in handling the uncertainty in the considered model and numerically validate the theoretical results established in the previous sections. To solve the stochastic Galerkin system \eqref{sysmat}-\eqref{sysbound}, we employ a shooting method based on the solver \texttt{fde12} \cite{garrappa2010linear} for systems of fractional differential equations. Since the underlying method requires an appropriate initial velocity (first-order derivative), we use the differential evolution algorithm to determine the unknown initial velocity so that the prescribed boundary condition is satisfied~\cite{PulchSingh2026}. All numerical experiments were performed in MATLAB on a computer equipped with an 11th Gen Intel(R) Core(TM) i5-1135G7 @ 2.40 GHz, 2419 MHz, 4 Core(s) using the Microsoft Windows 11 operating system.\par 
We begin with a numerical illustration of Lemma~\ref{auxiliary} by examining the following two cases of the random coefficient function $a$ in the problem \eqref{fde}:\\

\textbf{Case 1:} $\qquad a(\xi)=2+(\xi-1)^5,$\\

\textbf{Case 2:}$\qquad a(\xi)=
\begin{cases}
\xi^{3/2}, & \text{if } \xi \geq 0,\\[2mm]
-(-\xi)^{3/2}, & \text{if } \xi < 0,
\end{cases}$

where $\xi$ is uniformly distributed in the interval $[1,2]$ and $[-1,1]$ in Case 1 and Case 2, respectively.
\begin{figure}[!ht]
\begin{center}
\centering
\subfigure[Case 1]{%
\includegraphics[scale=0.44]{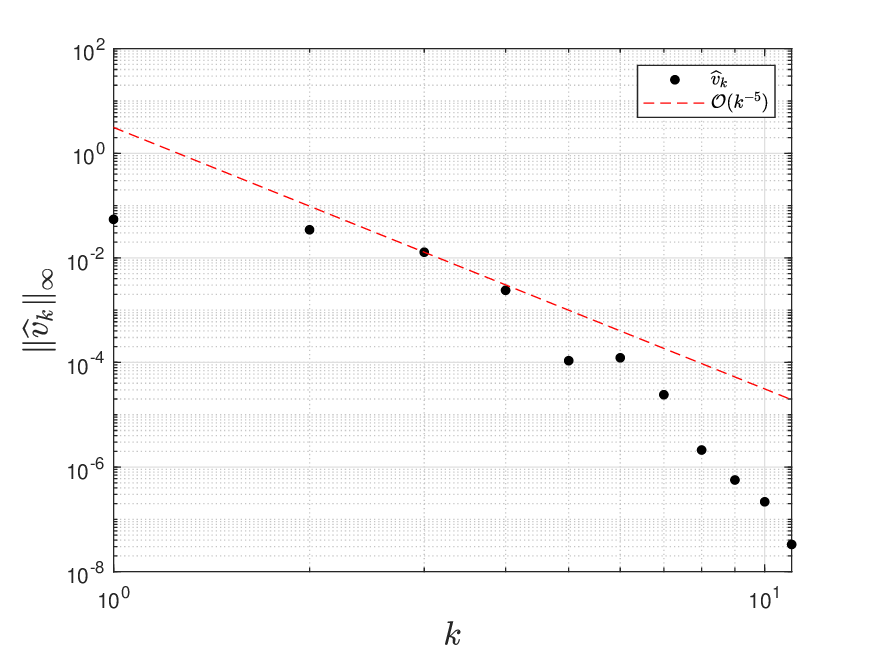}
\label{fig1Lemma3}}
\quad
\subfigure[Case 2]{%
\includegraphics[scale=0.44]{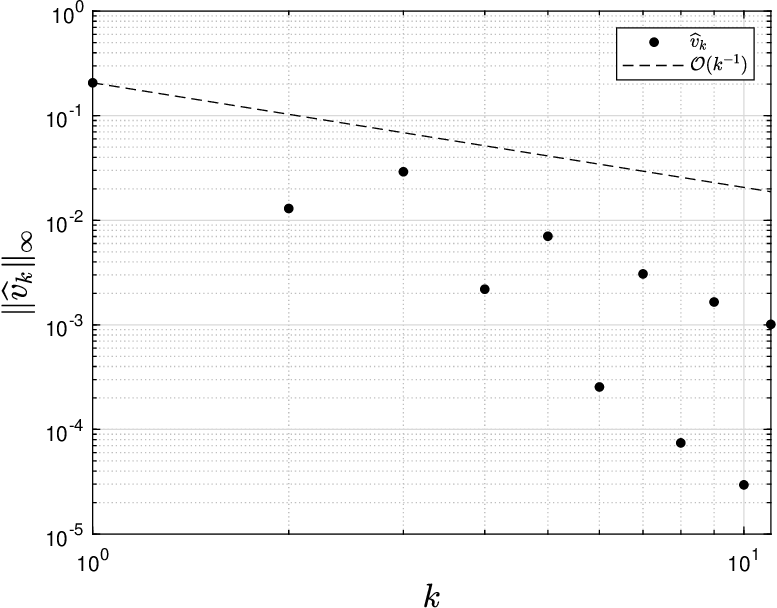}
\label{fig2Lemma3}}
\caption{Validation of Lemma \ref{auxiliary} for $\alpha = 1.5$ in Case 1 and Case 2.}
\label{fig_Lemma3}
\end{center}
\end{figure}
Moreover, we consider the boundary conditions in \eqref{bound} as
\[
y(0,\xi)=2, \qquad y(1,\xi)=4,
\]
for all $\xi$ in both cases.
Although the gPC expansion in \eqref{gpc} involves infinitely many coefficients, for the numerical illustration we compute only the coefficients $\hat{v}_k$ for $k=1,\ldots,11$, using a 60-point Gauss-Legendre quadrature rule for the integrals that define the coefficients with the value of $\alpha = 1.5$ and $T = 1$. This is obtained by substituting the expansion of \(y(t,\xi)\) from Equation~\eqref{sol} into Equation~\eqref{gpc}, taking the inner product with the chosen basis functions, and applying the given boundary conditions. This number of coefficients is sufficient to illustrate the theoretical decay behavior proved in Lemma~\ref{auxiliary}, although the computation can readily be extended to larger values of $k$. 

In Case~1, the coefficient function $a$ is smooth and, in particular, belongs to the Sobolev spaces $W^{5,\infty}(1,2)$. The constant $C$ in Lemma~\ref{auxiliary} is estimated from the computed coefficients $\hat{v}_k$ so that the corresponding theoretical bound is as sharp as possible, which results in $C=3.1186$. 
Figure~\ref{fig_Lemma3}(a) shows that the maximum value of the computed gPC coefficients $\hat{v}_k$ over the time interval remains below the reference decay rate $\mathcal{O}(k^{-5})$, which is in agreement with the theoretical estimates presented in Lemma~\ref{auxiliary}. 

In Case~2, the coefficient function $a$ has substantially lower regularity because of its singular derivative at \(\xi=0\) and, in particular, belongs to $W^{1,\infty}(-1,1)$. The computed gPC coefficients exhibit a slower and oscillatory decay, as shown in Figure~\ref{fig_Lemma3} (b). The corresponding reference constant obtained from the computed coefficients is $C=0.2063$, and the numerical results exhibit a decay below \(\mathcal{O}(k^{-1})\). 

Next, we numerically validate the superlinear convergence result for the gPC expansion coefficients established in Lemma~\ref{superlinear}, using the data from Case~1. For the numerical computation, the infinite sum in Lemma \ref{superlinear} is approximated by a truncation at $k=1000$. Define
\begin{equation*}
    V_{n}(t) := (n+1)\sum_{k = n+1}^{1000} \abs{\hat{v}_k(t)}. 
\end{equation*}
We compute $\hat{v}_k(t)$ at $t=0.5$ and $t=1.5$ using Gauss--Legendre quadrature with fractional order $\alpha=1.5$. 
From Figure \ref{fig_Lemma5}, one can observe that $V_{n}$ decreases rapidly as $n$ increases, which validates the superlinear convergence behaviour established in Lemma \ref{superlinear}.\par 
\begin{figure}[!ht]
\begin{center}
\centering
\includegraphics[scale=0.5]{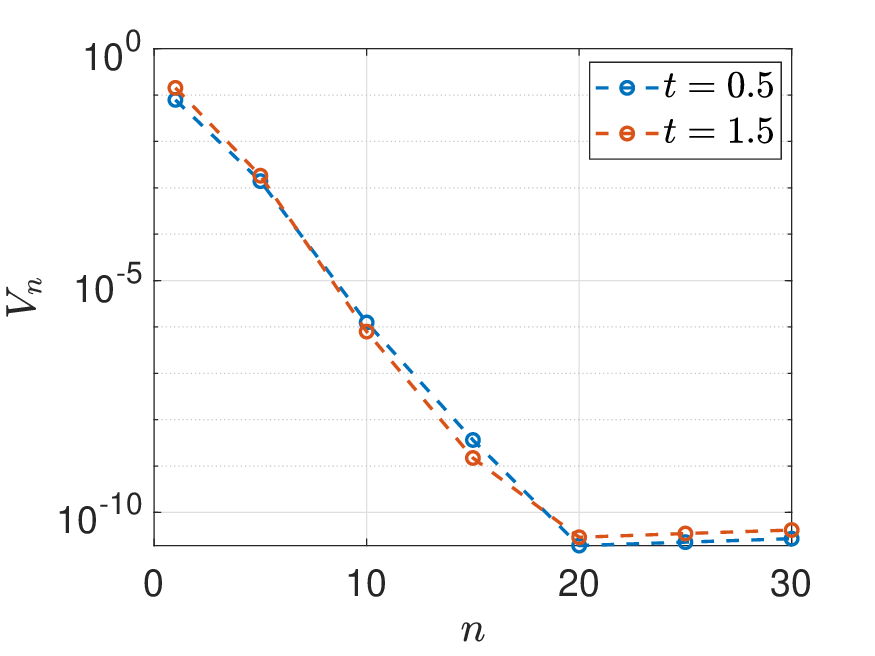}
\caption{Validation of Lemma \ref{superlinear} for $\alpha = 1.5$ at $t = 0.5$ and $t = 1.5$ in Case 1.}
\label{fig_Lemma5}
\end{center}
\end{figure}

Next, we examine the estimates established in Lemmas~\ref{abound} and \ref{lambdabound} using the same data as in Cases~1 and 2.  The constants $C_{a}$ in these cases can be easily computed as $3$ and $1$, respectively. For $n=3$, let $A_1$ and $A_2$ denote the corresponding stochastic Galerkin matrices for Cases~1 and 2, respectively. These matrices, with their entries rounded to four decimal places, are given by
\begin{center}
\begin{minipage}{0.48\textwidth}
\centering
\[
A_{1} =
\begin{pmatrix}
2.1667 & 0.2062 & 0.1331 & 0.0525  \\
0.2062 & 2.2857 & 0.2305 & 0.1273  \\
0.1331 & 0.2305 & 2.2619 & 0.2134  \\
0.0525 & 0.1273 & 0.2134 & 2.2525 
\end{pmatrix},
\]
\end{minipage}
\hfill
\begin{minipage}{0.48\textwidth}
\centering
\[
A_{2}=
\begin{pmatrix}
0      & 0.4945 & 0      & 0.0697  \\
0.4945 & 0      & 0.5035 & 0        \\
0      & 0.5035 & 0      & 0.4596   \\
0.0697& 0      & 0.4596 & 0       
\end{pmatrix}.
\]
\end{minipage}
\end{center}
It can be observed that the absolute values of the entries of $A_1$ and $A_2$ are bounded by the corresponding constants $3$ and $1$, respectively, in agreement with the assertion of Lemma \ref{abound}. Furthermore, to illustrate Lemma \ref{lambdabound}, we plot the spectra of the matrices of the stochastic Galerkin system for different values of $n$ in both cases. As shown in Figure \ref{fig_Lemma7}, the computed spectra remain within the bounds established in Lemma \ref{lambdabound}. We can also see the sharper bounds $\lambda \in [2,3]$ and $\lambda \in [-1,1]$ due to $a \in [2,3]$ and $a \in [-1,1]$, respectively.\par 
\begin{figure}[!ht]
\begin{center}
\centering
\subfigure[Case 1]{%
\includegraphics[scale=0.44]{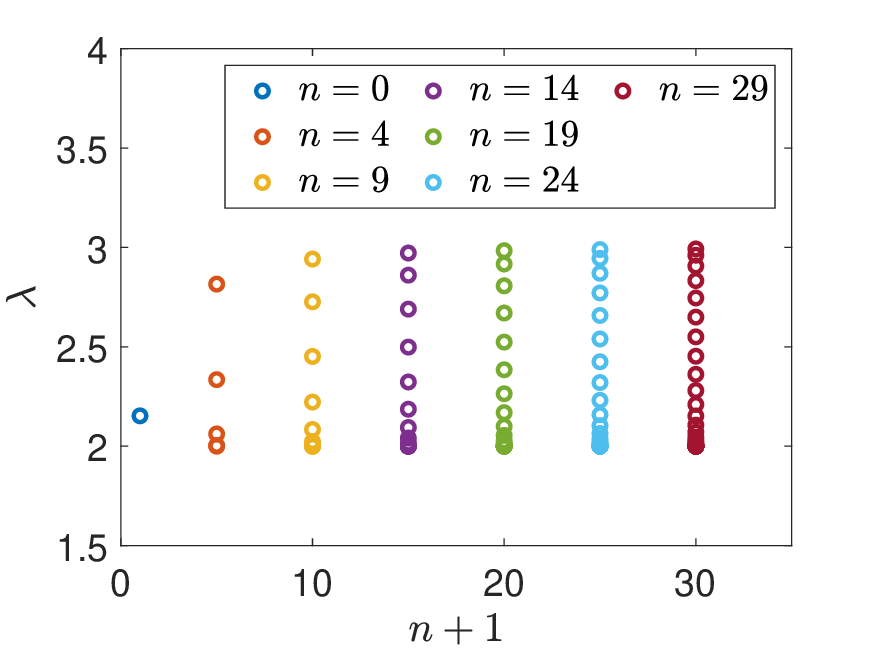}
\label{fig1Lemma7}}
\quad
\subfigure[Case 2]{%
\includegraphics[scale=0.44]{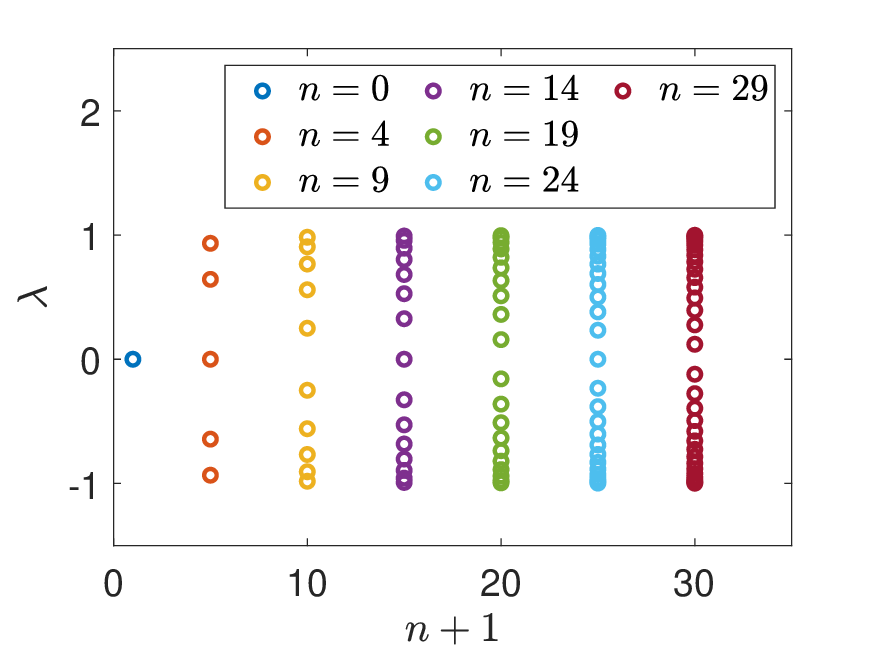}
\label{fig2Lemma3}}
\caption{Spectra of the matrices \(A \in \mathbb{R}^{(n+1)\times (n+1)}\) for different numbers \(n+1\) of gPC coefficients in the Case $1$ and Case $2$.}
\label{fig_Lemma7}
\end{center}
\end{figure}

Finally, we investigate the validation of the main convergence result established in Section \ref{converge} and the behavior of the numerical solution of the stochastic Galerkin system \eqref{sysmat} developed in Section \ref{sgmethod}. For this purpose, we consider random boundary data, i.e., the boundary conditions also depend on a random parameter $\xi$, which provides a more general setting than the data considered in the preceding experiments. We define the following form of $a(\xi)$:
\begin{equation}\label{num_eq1}
    a(\xi)=\xi^{p\alpha},
    \tag{11}
\end{equation} 
where $\xi$ is uniformly distributed on \([\xi_{\min},\xi_{\max}]\), i.e., $\xi \sim U(\xi_{\min},\xi_{\max})$ and $p$ is any positive real number.
Moreover, we consider the boundary conditions 
$$y(0,\xi) = c_{1}+c_{2}\xi^{5} \quad  \text{and} \quad y(1,\xi) = d_{1}+d_{2}\xi^{5}.$$ In view of Equation \eqref{sysbound}, we compute the boundary conditions for the stochastic Galerkin system. In particular, for \(n=2\), \(c_{1}=2\), \(c_{2}=1\), \(d_{1}=4\), \(d_{2}=3\), and \(\xi \sim U(0,2)\) the boundary conditions become
\[
\vb v^{(2)}(0) = [7.3333,\,6.5982,\,4.2591]^{'}
\quad \text{and} \quad
\vb v^{(2)}(1) = [20.0000,\,19.7948,\,12.7775]^{'}.
\]
These boundary conditions are independent of
\(a(\xi)\) given by Equation \eqref{num_eq1} and thus independent of the value $p$. To investigate the convergence of the stochastic Galerkin approximation, we define a measure motivated by the property of a Cauchy sequence
\begin{equation*}
E_{i,n}^{(2)} 
:=
\left(
\sum_{k=0}^{n}
\bigl(v_i(t_{k})-v_{i-1}(t_k)\bigr)^2
\right)^{1/2},\quad i = 1,2,\ldots,n.
\end{equation*}

\begin{figure}[!ht]
\begin{center}
\centering
\subfigure[$\alpha = 1.1$]{%
\includegraphics[scale=0.32]{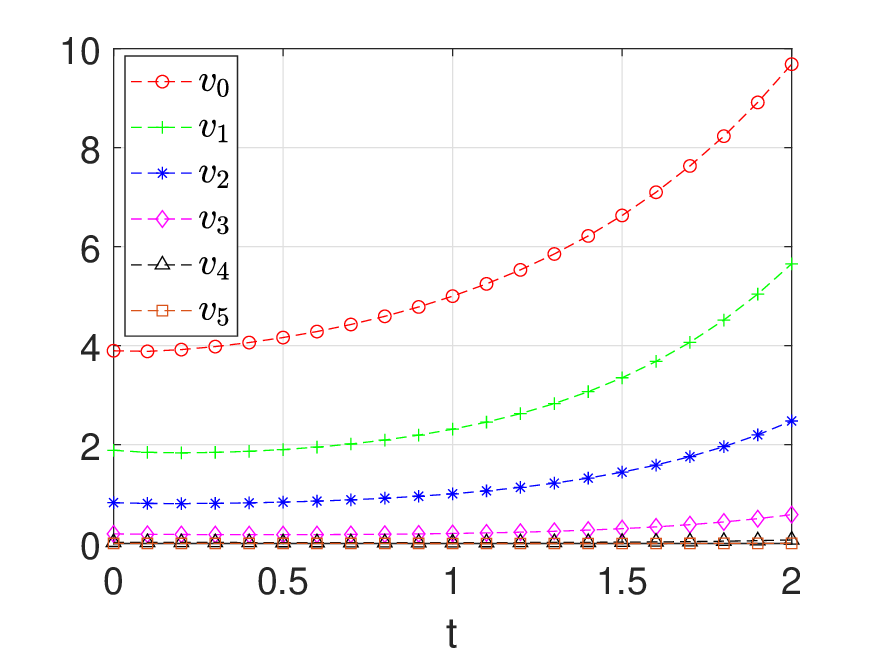}
\label{fig1:1}}
\quad
\subfigure[$\alpha = 1.5$]{%
\includegraphics[scale=0.32]{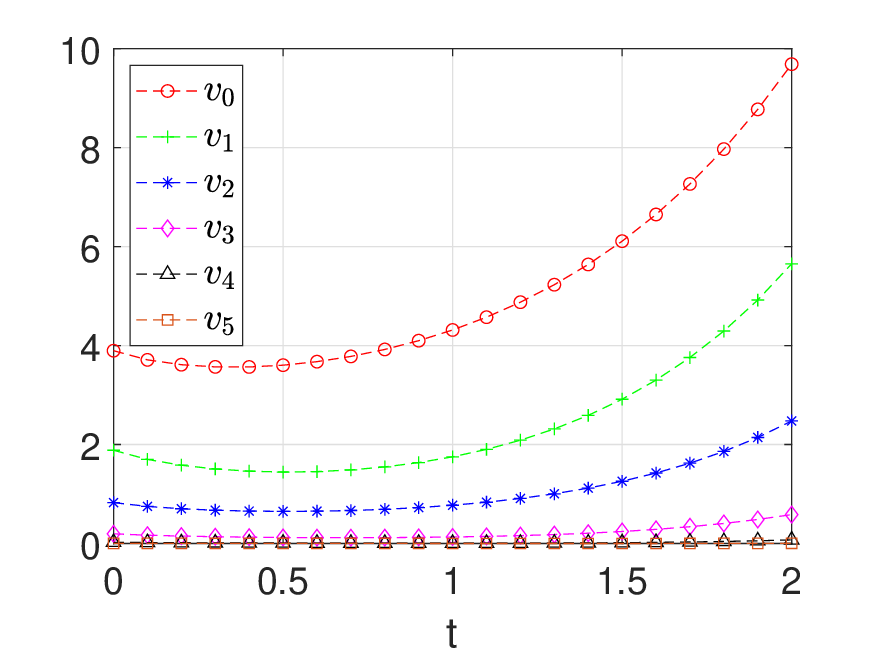}
\label{fig1:2}}
\quad
\subfigure[$\alpha = 1.9$]{%
\includegraphics[scale=0.32]{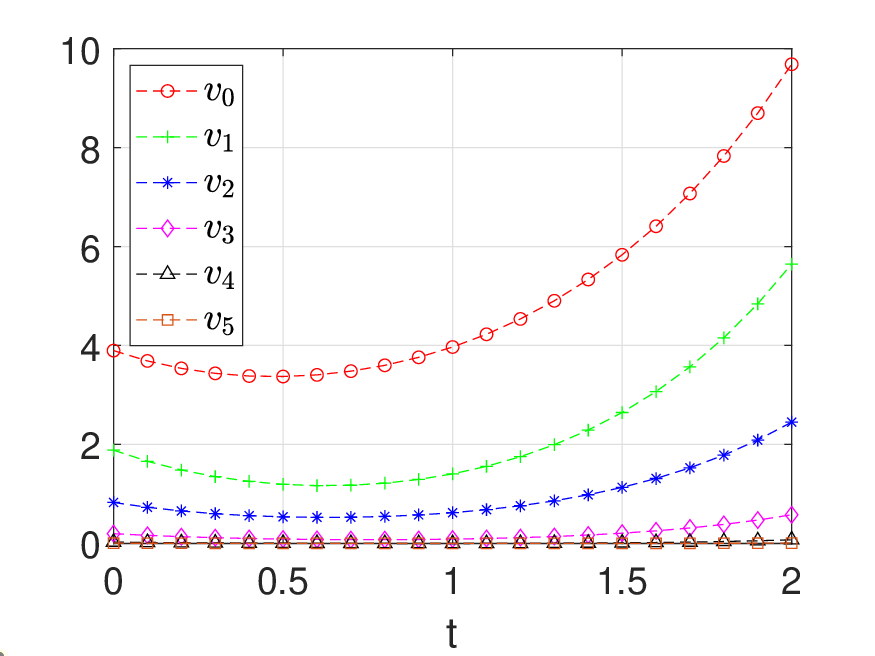}
\label{fig1:3}}
\caption{Numerical approximation for coefficient functions of the stochastic Galerkin system for $n=5,~\xi_{min} = 0.5,~\xi_{max} = 1.5$ with different values of $\alpha$ using $a(\xi)=\xi^{p\alpha}$ with $p = 1$ .}
\label{fig:case1}
\end{center}
\end{figure}

\begin{figure}[!ht]
\begin{center}
\centering
\subfigure[$\alpha = 1.1$]{%
\includegraphics[scale=0.32]{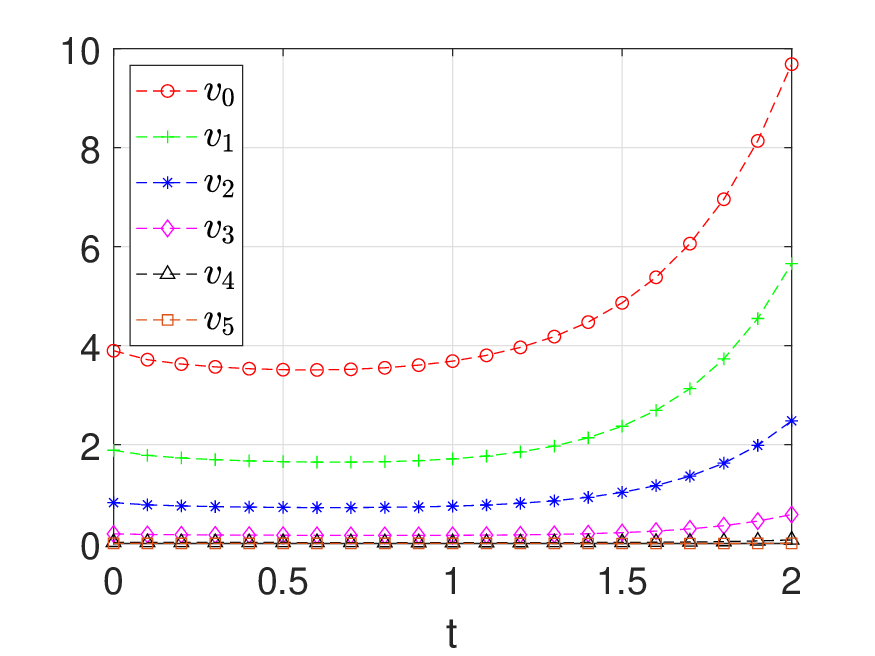}
\label{fig2:1}}
\quad
\subfigure[$\alpha = 1.5$]{%
\includegraphics[scale=0.32]{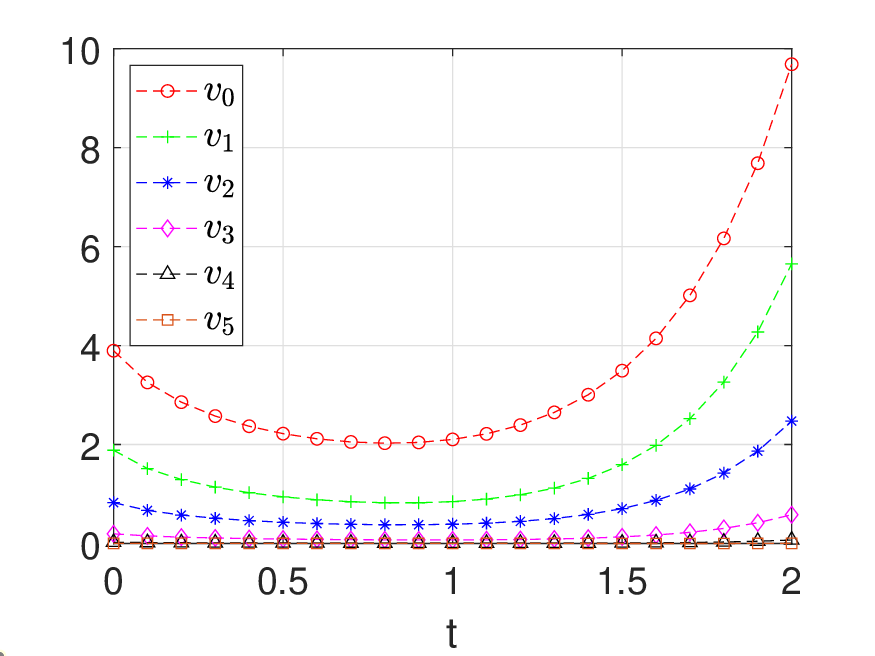}
\label{fig2:2}}
\quad
\subfigure[$\alpha = 1.9$]{%
\includegraphics[scale=0.32]{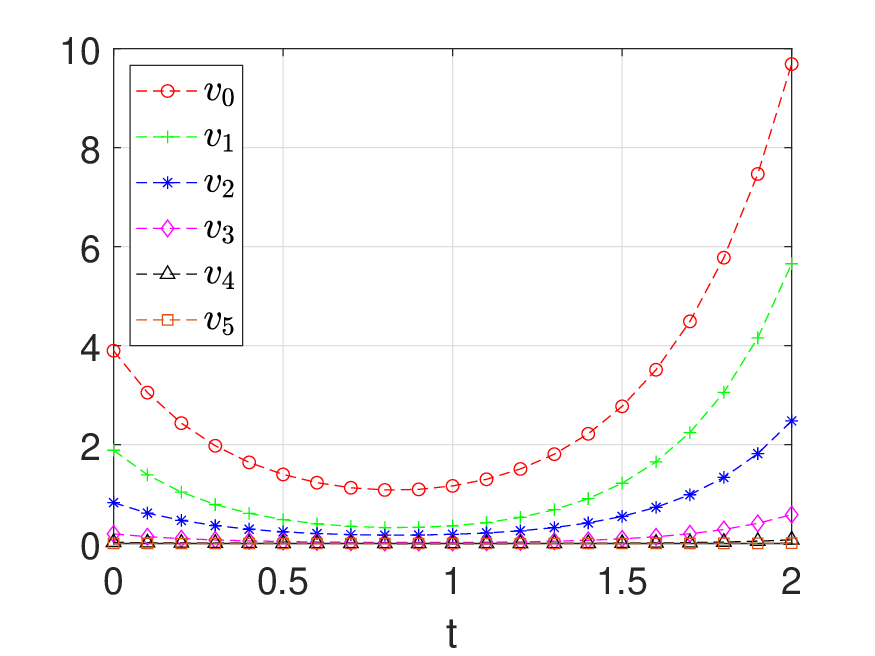}
\label{fig2:3}}
\caption{Numerical approximation for coefficient functions of the stochastic Galerkin system for $n=5,~\xi_{min} = 2,~\xi_{max} = 3$ with different values of $\alpha$ using $a(\xi)=\xi^{p\alpha}$ with $p = 1$.}
\label{fig:case2}
\end{center}
\end{figure}

Figures \ref{fig:case1} and \ref{fig:case2} show the numerical solutions of the stochastic Galerkin system for different values of \(\alpha\) with $n = 4$, for \(\xi_{min} = 0.5,~ \xi_{max} = 1.5\) and \(\xi_{min} = 2,~ \xi_{max} = 3\), respectively. In both cases, the coefficient \(v_{0}(t)\), which represents the mean component of the solution, has the largest magnitude, while the higher-order coefficients \(v_{k}(t)\), \(1 \leq k \leq 4\), are considerably smaller. The relatively small magnitude of the higher-order coefficients indicates that the variance of the solution is moderate because the solution is smoothly dependent on $\xi$. This behavior is expected since \(a(\xi)=\xi^\alpha\) is smooth with respect
to \(\xi\) on the considered intervals, which are bounded away from zero. Furthermore, varying the fractional order \(\alpha\) slightly changes the magnitudes of the coefficient functions while preserving the overall qualitative behavior of the solution. \par

\begin{figure}[!ht]
\begin{center}
\centering
\subfigure[$\xi_{min} = 0,\xi_{min} = 2,p =2$]{%
\includegraphics[scale=0.44]{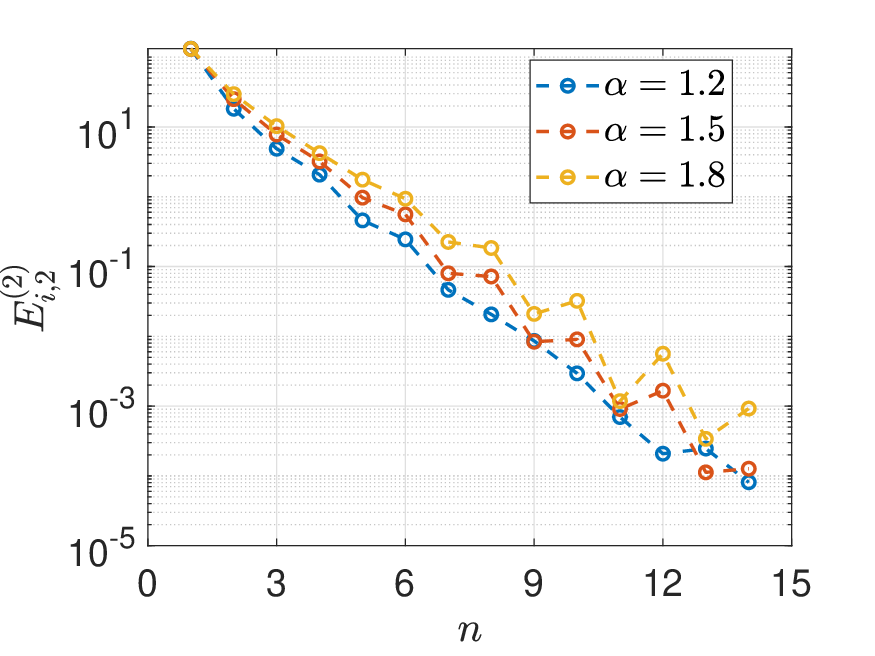}
\label{fig1thm1}}
\quad
\subfigure[$\xi_{min} = 0,\xi_{min} = 3, p =1$]{%
\includegraphics[scale=0.44]{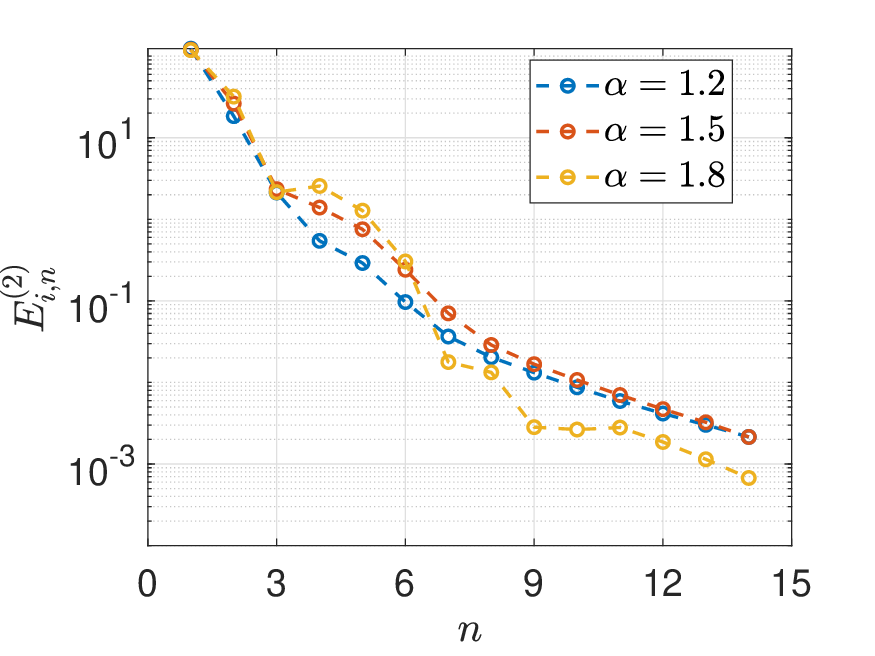}
\label{fig2thm1}}
\caption{Convergence plot for different values of \(\alpha\) with \(c_{1}=2\), \(c_{2}=0\), \(d_{1}=4\), and \(d_{2}=4\).}
\label{fig_thm1}
\end{center}
\end{figure}

Figure \ref{fig_thm1} displays the convergence behavior of the stochastic Galerkin system for different values of $\alpha,~\xi_{min},~\xi_{max},~c_1,~c_2,~d_{1},~d_{2}$ and $p$. In all the considered cases, $E_{i,n}^{(2)}$ decreases as $n$ increases, which provides the numerical confirmation of the convergence result established in Section~\ref{converge}.

\begin{figure}[!ht]
\begin{center}
\centering
\subfigure[$\xi_{min} = -1,~\xi_{max} = 0$]{%
\includegraphics[scale=0.44]{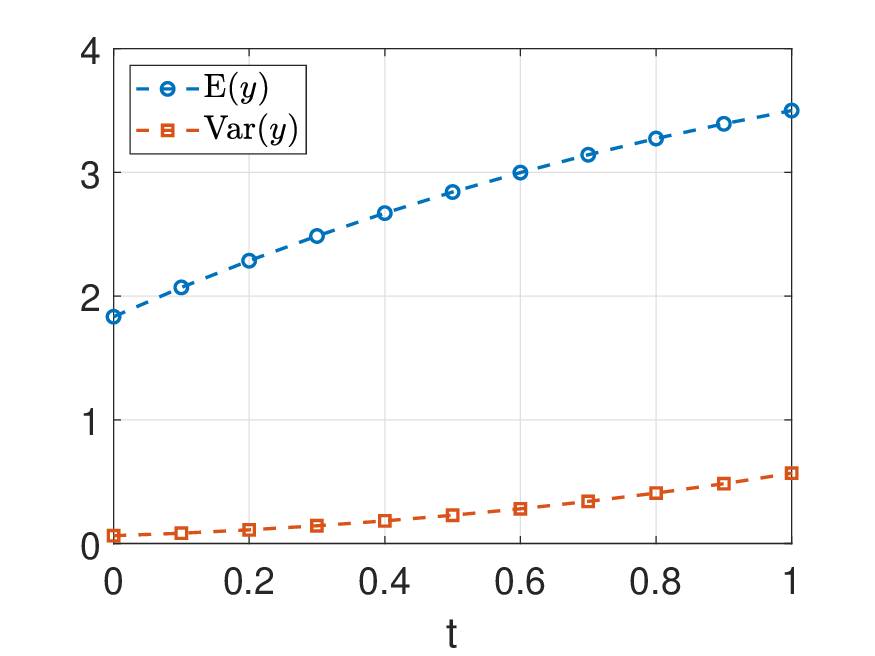}
\label{figmean:1}}
\quad
\subfigure[$\xi_{min} = 0,~\xi_{max} = 1$]{%
\includegraphics[scale=0.44]{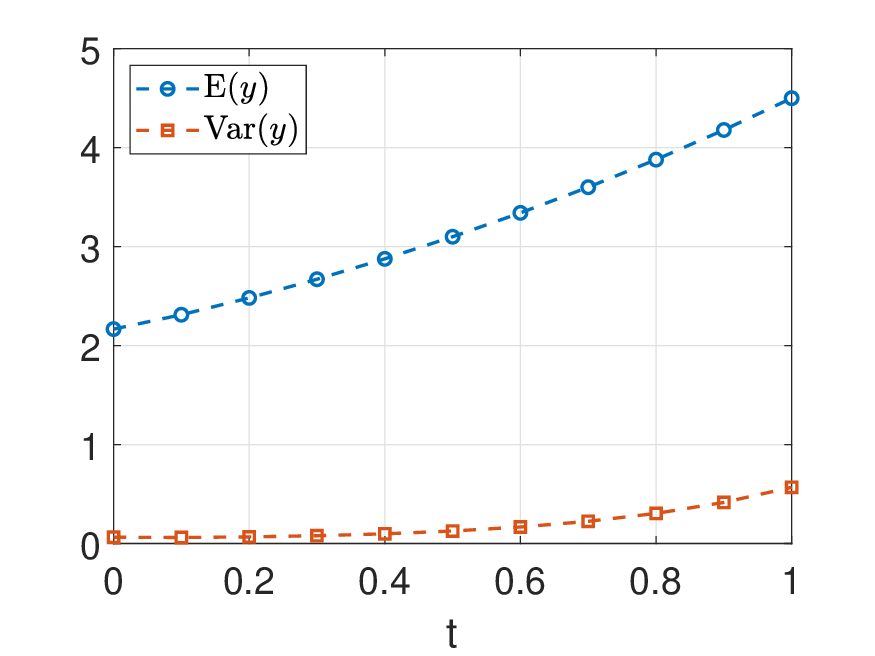}
\label{figmean:2}}
\caption{Numerical approximations of the mean and variance obtained using the stochastic Galerkin system with $n=10$ and $\alpha=1.5$.}
\label{fig:mean_var}
\end{center}
\end{figure}
The coefficient functions of the stochastic Galerkin system~\eqref{sysmat} provide approximations of the expected value and the variance of the random process. Specifically, we obtain
\begin{equation}\label{mean}
\mathbb{E}(y(t,\xi)) \approx v_{0}(t),
\qquad
\operatorname{Var}(y(t,\xi)) \approx \sum_{i=1}^{n} v_{i}^{2}(t).\tag{12}
\end{equation}
We compute the mean and variance defined by Equation \eqref{mean} for \(a(\xi)=\xi^{p\alpha}\) and different sets of parameters $\xi_{max}$ and \(\xi_{max}\), with fixed values \(c_{1}=2\), \(c_{2}=1\), \(d_{1}=4\), \(d_{2}=3\), and \(\alpha = 1.5\), while setting \(p=\frac{1}{\alpha}\). Figure~\ref{fig:mean_var} shows the numerical approximations of the expected value and variance obtained from~\eqref{mean} for the considered data with \(n=10\), for \(\xi_{min} =-1,~\xi_{max}=0\) and \(\xi_{min}=0,~\xi_{min}=1\). From Figure~\ref{fig:mean_var}, one can observe that the variance remains relatively small compared with the magnitude of the mean in both cases. Although the two cases exhibit a different behavior of the variance near the boundary, this difference is expected because the stochastic Galerkin system depends on the boundary conditions, which themselves depend on the random variable \(\xi\). Consequently, the choice of the parameters \(a\) and \(b\) influences both the mean of the solution and the associated uncertainty over the interval.

\section{Conclusion}
In this work, we considered the stochastic Galerkin method for two-point fractional boundary value problems with random parameters. Our analysis shows that, for the convergence study, it is not necessary to assume the required properties of the stochastic coefficients directly. We instead start with regularity assumptions on the given input data and derive the corresponding properties of the stochastic coefficients from them. In this way, the assumptions used in the convergence analysis are connected directly to the data of the original problem. \par
The numerical experiments support the theoretical analysis and illustrate the effect of uncertainty on the solution. In particular, they show how random coefficients and random boundary conditions influence the mean and variance of the solution. The results also indicate that the behavior of the uncertainty can depend significantly on the choice of boundary data. The present work is restricted to linear two-point fractional boundary value problems with a finite-dimensional representation of the random inputs. It would therefore be interesting to extend the analysis to nonlinear problems and to models involving a larger number of random variables.\par

\section*{Declarations}
\subsection*{Data Availability Statement}
No data is available for this manuscript.
\subsection*{Conflict of Interest}
The author(s) declared no potential conflicts of interest with respect to the research, authorship, and/or publication of this article.
\section*{Acknowledgments}
The third author acknowledges the Anusandhan National Research Foundation (ANRF), Government
of India, for supporting this work through the Prime Minister Early Career Research
Grant (PM-ECRG) under Grant No. ANRF/ECRG/2024/002135/PMS.

\printbibliography


\end{document}